\documentclass[10pt]{article}
\usepackage{amsmath, amsthm}\usepackage{enumerate}\usepackage{amssymb}\usepackage{bbold}
\usepackage{color}
\newtheorem{theorem}{Theorem}[section]
\newtheorem{definition}{Definition}[section] 
\newtheorem{exam}{Example}[section]
\newtheorem{rem}{Remark}[section]

\newtheorem{quest}{Question}[section]
\newtheorem{prop}{Proposition}[section]
\newtheorem{lem}{Lemma}[section]
\newtheorem{cor}{Corollary}[section]

\DeclareMathOperator{\convo}{\xrightarrow[]{o}}
\DeclareMathOperator{\convuo}{\xrightarrow[]{uo}}
\DeclareMathOperator{\convw}{\xrightarrow[]{w}}

\DeclareMathOperator{\convnX}{\xrightarrow[]{\|\cdot\|_X}}
\DeclareMathOperator{\convnY}{\xrightarrow[]{\|\cdot\|_Y}}
\DeclareMathOperator{\convtau}{\xrightarrow[]{\tau}}
\DeclareMathOperator{\convvars}{\xrightarrow[]{\varsigma}}

\DeclareMathOperator{\convvarsh}{\xrightarrow[]{\hat{\varsigma}}}

\begin{document}

\title{{\bf $o\tau$-Continuous, Lebesgue, KB, and Levi operators between vector lattices and topological vector spaces}}
\maketitle\author{\centering{{Safak Alpay, Eduard Emelyanov, Svetlana Gorokhova\\ }
\abstract{We investigate $o\tau$-continuous/bounded/compact and Lebesgue operators from vector lattices to topological vector spaces;
the Kantorovich-Banach operators between locally solid lattices and topological vector spaces; and the Levi operators from locally solid lattices to vector lattices. 
The main idea of operator versions of notions related to vector lattices lies in redistributing  topological and order properties of 
a topological vector lattice between the domain and range of an operator under investigation. Domination properties for these classes of operators are
studied. 
}

{\bf{keywords:}} {\rm topological vector space, locally solid lattice, Banach lattice, order convergence, domination property, adjoint operator}

\large

\section{Introduction and preliminaries}

In the present paper, all vector spaces are supposed to be real, operators linear, vector topologies Hausdorff, and vector lattices Archimedean. 

For any vector lattice $X$, the Dedekind complete vector lattice of all order bounded linear functionals on $X$ is called {\em order dual} of $X$ and is
denoted by $X^\sim$. Recall that the order continuous part  $X^\sim_n$  is a band of $X^\sim$. Some vector lattices may have trivial order duals, for example 
$X^\sim=\{0\}$ whenever $X=L_p[0,1]$ with $0<p<1$ (cf. \cite[Thm.5.24]{AB1}). 

If $T$ is an operator from a vector space $X$ to a vector space $Y$, 
the {\em algebraic adjoint} $T^{\#}$ is an operator from the algebraic dual $Y^{\#}$ to $X^{\#}$, defined by $(T^{\#}f)(x):=f(Tx)$ for all $x\in X$ and all $f\in Y^{\#}$. 
In the case when $X$ and $Y$ are vector lattices and $T$ is order bounded, the restriction $T^\sim$ of $T^{\#}$ to the order dual $Y^\sim$ of $Y$ is called the 
{\em order adjoint} of $T$. The operator  $T^\sim: Y^\sim \to X^\sim$ is not only order bounded, but even order continuous 
(cf. \cite[Thm.1.73]{AB2}). Clearly, $T^\sim: Y_n^\sim \to X_n^\sim$ when $T$ is order continuous. 

In the case when  $(X,\varsigma)$ and $(Y,\tau)$ are topological vector spaces and $T$ is continuous, the restriction $T'$ of $T^{\#}$ to the 
{\em topological dual} $Y'$ (=  the collection of all $\varsigma$-continuous linear  functionals on $Y$) is called the {\em topological adjoint} of $T$.
For every locally solid lattice $(X,\varsigma)$, we have $X'\subseteq X^\sim$ since $\varsigma$-continuous functionals are bounded on 
$\varsigma$-bounded subsets and hence are order bounded. The topological dual $X'$ of a locally convex-solid lattice $(X,\varsigma)$ is 
an ideal of $X^\sim$ (and hence $X'$ is Dedekind complete) (cf. \cite[Thm.3.49]{AB2}). Every Fr{\'e}chet lattice $(X,\varsigma)$ 
satisfies $X'=X^\sim$ (cf. \cite[Thm.5.23]{AB1}). 

A net $(x_\alpha)_{\alpha\in A}$ in a vector lattice $X$ is said to be:
\begin{enumerate}
\item[$a)$] \  
{\em order convergent} ({\em $o$-convergent}) to $x\in X$, if there exists a net $(z_\beta)_{\beta\in B}$ in $X$ such that $z_\beta\downarrow 0$ and, 
for any $\beta\in B$, there exists $\alpha_\beta\in A$ with $|x_\alpha-x|\leq z_\beta$ for all $\alpha\geq\alpha_\beta$. In this case, 
we write $x_\alpha\convo x$$;$
\item[$b)$] \ 
{\em $uo$-convergent} to $x\in X$, if $|x_\alpha-x|\wedge u\convo 0$ for every $u\in X_+$.
\end{enumerate}

A {\em locally solid lattice} $(X,\varsigma)$ is called: 
\begin{enumerate}
\item[$c)$] \  
{\em Lebesgue}/$\sigma$-{\em Lebesgue} if $x_\alpha\downarrow 0$ implies $x_\alpha\convvars 0$ for every net/sequence $x_\alpha$ in $X$.
\item[$d)$] \  
{\em pre-Lebesgue} if $0\le x_n\uparrow\le x$ in $X$ implies that $x_n$ is a $\varsigma$-Cauchy sequence in $X$.
\item[$e)$] \  
{\em Levi}/$\sigma$-{\em Levi} \ if every increasing $\varsigma$-bounded net/sequen\-ce in $X_+$ has a supremum in $X$.
\item[$f)$] \  
{\em Fatou}/$\sigma$-{\em Fatou} if the topology $\varsigma$ has a base at zero consisting of order/$\sigma$-order closed solid sets.
\end{enumerate}
The assumption $x_\alpha\downarrow 0$ on the net $x_\alpha$ in $c)$ can be replaced by $x_\alpha\convo 0$.
A Lebesgue lattice $(X,\varsigma)$ is Dedekind complete iff order intervals of $X$ are $\varsigma$-complete \cite[Prop.3.16]{Tay1}.
Every Lebesgue lattice is pre-Lebesgue \cite[Thm.3.23]{AB1} and Fatou \cite[Lem.4.2]{AB1}. Furthermore, $(X,\varsigma)$ is pre-Lebesgue iff 
the topological completion $(\hat{X},\hat{\varsigma})$ of $(X,\varsigma)$ is Lebesgue \cite[Thm.3.26]{AB1} iff $(\hat{X},\hat{\varsigma})$ is 
pre-Lebesgue and Fatou \cite[Thm.4.8]{AB1}. By \cite[Thm.3.22]{AB1}, $d)$ is equivalent to each of the following two conditions:
\begin{enumerate}
\item[$d')$] \  
if $0\le x_\alpha\uparrow\le x$ holds in $X$, then $x_\alpha$ is a $\varsigma$-Cauchy net;
\item[$d'')$] \  
every order bounded disjoint sequence in $X$ is $\varsigma$-conver\-gent to zero.
\end{enumerate}
The next well known fact follows directly from $d')$.

\begin{prop}\label{tau compl pre Leb is Ded compl}
Every $\varsigma$-complete pre-Lebesgue lattice $(X,\varsigma)$ is Dedekind complete.  
\end{prop}

A normed lattice $(X,\|\cdot\|)$ is called {\em Kantorovich-Banach} or {\em $KB$-space} if every norm bounded upward directed 
set in $X_+$ converges in the norm. Each $KB$-space is a Levi lattice with order continuous complete norm; each order continuous Levi normed lattice is Fatou; 
and each Levi normed lattice is Dedekind complete. Lattice-normed versions of $KB$-spaces were recently studied in \cite{AEEM1,AEEM2,AGG}. 

\begin{enumerate}
\item[$g)$] \  
We call a locally solid lattice $(X,\varsigma)$ by a {\em $KB$/$\sigma$-$KB$ lattice} if every increasing $\varsigma$-bounded net/sequ\-ence in $X_+$ is $\varsigma$-convergent.
\end{enumerate}
Clearly, each $KB$/$\sigma$-$KB$ lattice is Levi/$\sigma$-Levi and each Levi/$\sigma$-Levi lattice is Dedekind complete/$\sigma$-complete.

Recall that a continuous operator $T$:
\begin{enumerate}
\item[$h)$] \
between two Banach spaces is said to be {\em Dunford-Pettis} if $T$ takes weakly null sequences to norm null sequences. 
It is well known that every weakly compact operator on $L_1(\mu)$ is Dunford-Pettis and that an operator is Dunford-Pettis iff it takes 
weakly Cauchy sequences to norm convergent sequences \cite[Thm.5.79]{AB2}. 
\item[$i)$] \
from a Banach lattice $X$ to a Banach space $Y$ is called {\em $M$-weakly compact} if
$\|Tx_n\|_Y\to 0$ holds for every norm bounded disjoint sequence $x_n$ in $X$.
\item[$j)$] \
from a Banach space $Y$ to a Banach lattice $X$ is called {\em $L$-weakly compact} 
whenever $\|x_n\|_X\to 0$ holds for every disjoint sequence $x_n$ in the solid hull of $T(U_Y)$, 
where $U_Y$ is the closed unit ball of $Y$.
\end{enumerate}
An operator $T$ from a vector lattice $X$ to a topological vector space $(Y,\tau)$ is called
\begin{enumerate}
\item[$k)$] \
{\em $o\tau$-continuous}/{\em $\sigma{}o\tau$-continuous} if $Tx_\alpha\convtau 0$ for every net/sequence $x_\alpha$ such that $x_\alpha\convo 0$ \cite{JAM}.
Replacement of $o$-null nets/sequences by $uo$-null ones above gives the definitions of {\em $uo\tau$-continu\-ous}/{\em $\sigma{}uo\tau$-continu\-ous} operators.
\end{enumerate}
$L$-/$M$-weakly compact operators are weakly compact (cf. \cite[Thm.5.61]{AB2}) and the norm limit of a sequence of $L$-/$M$-weakly 
compact operators is again $L$-/$M$-weakly compact (cf. \cite[Thm.5.65]{AB2}).
For further unexplained terminology and notions, we refer to \cite{AB1,AB2,AAT1,AAT2,GTX,MN,Wick,Za}.

Various versions of Banach lattice properties like a property to be a $KB$-space were investigated recently
(see, e.g., \cite{AP,AEGp,AG,AM,AEEM1,AEEM2,BA,DEM1,DEM1a,EEG,EGK,EM0,EM1,EM2,EGZ,EGOU,GTX,JAM,Tay1,Tay2,TA}). 
In the present paper we continue the study of operator versions of several topological/order properties, focusing on locally solid lattices. 
The main idea behind operator versions consists in a redistribution of topological and order properties of a topological vector lattice
between the domain and range of the operator under investigation (like in the case of Dunford-Pettis and $L$-/$M$-weakly compact operators). 
As the order convergence is not topological in general \cite{DEM1,Go}, the most important operator versions 
emerge when both $o$- and $\varsigma$-convergences are involved simultaneously. 

\begin{definition}\label{order-to-topology} {\em
Let $T$ be an operator from a vector lattice $X$ to a topological vector space $(Y,\tau)$. We say that:
\begin{enumerate}
\item[$(a)$] \  
$T$ is {\em $\tau$-Lebes\-gue}/{\em $\sigma\tau$-Lebesgue} if $Tx_\alpha\convtau 0$ for every net/ sequence $x_\alpha$ such that $x_\alpha\downarrow 0$;
$T$ is {\em quasi $\tau$-Lebes\-gue}/ {\em $\sigma\tau$-Lebesgue} if $Tx_\alpha$ is $\tau$-Cauchy for every net/sequence
$x_\alpha$ in $X_+$ satisfying $x_\alpha\uparrow\le x\in X$. If there is no confusion with the choice of the topology $\tau$ on $Y$, 
we call $\tau$-Lebesgue operators by {\em Lebesgue} etc.
\item[$(b)$] \  
$T$ is {\em $o\tau$-bounded}/{\em $o\tau$-compact} if $T[0,x]$ is a $\tau$-bounded/ $\tau$-totally bounded subset of $Y$ for each $x\in X_+$.
\end{enumerate}
If additionally $X=(X,\varsigma)$ is a locally solid lattice,
\begin{enumerate}
\item[$(c)$] \  
$T$ is {\em $KB$}/{\em $\sigma$-$KB$} if, for every $\varsigma$-bounded increasing net/sequence $x_\alpha$ in $X_+$,
there exists (not necessarily unique) $x\in X$ such that $Tx_\alpha\convtau Tx$.
\item[$(d)$] \  
$T$ is {\em quasi $KB$}/{\em quasi $\sigma$-$KB$} if $T$ takes $\varsigma$-bounded increasing nets/sequences in $X_+$ to $\tau$-Cauchy nets.
\end{enumerate}
If $X$ and $Y$ are vector lattices with $(X,\varsigma)$ locally solid,
\begin{enumerate}
\item[$(e)$] \  
$T$ is {\em Levi}/{\em $\sigma$-Levi} if, for every $\varsigma$-bounded increasing net/sequence $x_\alpha$ in $X_+$,
there exists (not necessarily unique) $x\in X$ such that $Tx_\alpha\convo Tx$.
\item[$(f)$] \  
$T$ is {\em quasi Levi}/{\em quasi $\sigma$-Levi} if $T$ takes $\varsigma$-bounded increasing nets/sequences in $X_+$ to $o$-Cauchy nets.
\end{enumerate}
Replacement of decreasing $o$-null nets/sequences by $uo$-null ones in $(a)$ and $o$-convergent ($o$-Cau\-chy) nets/sequen\-ces by ($uo$-Cau\-chy) 
$uo$-conver\-gent ones in $(e)$ and $(f)$ above gives the definitions of  {\em $uo\tau$-continu\-ous} and of ({\em quasi}) {\em $uo$-Levi operators} respectively. 
}
\end{definition}

In our approach, we focus on:
\begin{enumerate}
\item[$*$] \  
modification of nets/sets of operators domains in $(a)$, $(b)$, $(d)$, $(e)$, and $(f)$;
\item[$**$] \  
information which operators provide about convergences in their domains/ranges in $(c)$, $(e)$, and $(f)$. 
\end{enumerate}

\begin{rem}\label{c_00} 
{\em
\begin{enumerate}
\item[$a)$] \  
The identity operator in a locally solid lattice $(Y,\tau)$ is Lebesgue/$KB$/Levi iff $(Y,\tau)$ is Lebes\-gue/$KB$/Levi.
This motivates the terminology. Cle\-ar\-ly, every $o\tau$-continuous/$\sigma{}o\tau$-continuous operator is 
$\tau$-Lebesgue/$\sigma\tau$-Lebesgue. By Lemma \ref{PC1}, any positive operator to a locally solid lattice
is Lebesgue/$\sigma$-Lebesgue iff $T$ is $o\tau$-continuous/$\sigma{}o\tau$-continuous; and every positive 
Lebesgue operator to a locally solid lattice is quasi Lebesgue by Theorem \ref{Thm.3.23 from AB2}. 
\item[$b)$] \
Each regular operator from a vector lattice $X$ to a locally solid lattice $(Y,\tau)$ is $o\tau$-bounded by Proposition \ref{regular are otau-bounded}.
In the case of normed range space $Y$, $o\tau$-bounded operators are also known as {\em interval-bounded} (cf. \cite[Def.3.4.1]{MN}).
Like in \cite[Lem.3.4.2]{MN}, each $o\tau$-bounded operator $T$ from a vector lattice $X$ to a topological vector space $(Y,\tau)$ possesses {\em adjoint}
$T^{\circ} : Y' \to X^\sim$ given by $T^{\circ}y' = y'\circ T$ for all $y' \in Y'$.
\item[$c)$] \
In the case when $X$ is also a topological vector lattice, the $\tau$-continuity of operator $T$ is not assumed in $(b)$ of Definition \ref{order-to-topology}. For example,
the rank one discontinuous operator $Tx:=(\sum_{k=1}^{\infty}x_k)e_1$ in $(c_{00},\|\cdot\|)$ is $o\tau$-compact and $o\tau$-continuous yet not compact.
Each continuous operator $T$ from a discrete Dedekind complete locally convex Lebesgue lattice  to a topological vector space is $o\tau$-compact by the Kawai theorem (cf. \cite[Cor.6.57)]{AB1}.
Every Dunford-Pettis operator from a Banach lattice to a Banach space is $o$-weakly-compact (cf.  \cite[Thm 5.91]{AB2}).
\item[$d)$] \  
Each $KB$-operator is quasi $KB$; each quasi ($\sigma$-) $KB$-opera\-tor is quasi ($\sigma$-) Lebesgue;
and each continuous operator from a $KB$-space to a topological vector space is $KB$. 
It is well known that the identity operator $I$ in a Banach lattice is $KB$ iff $I$ is $\sigma${-}$KB$ iff $I$ is quasi $KB$. 
Proposition \ref{quasi-KB-vs-sigma-quasi-KB} shows that the notions of qua\-si $KB$ and qua\-si $\sigma$-$KB$ operator coincide.
The most important reason for using the term {\em $KB$-operator} for $(c)$ of Definition \ref{order-to-topology} is existence of limits of $\varsigma$-bounded increasing nets.
Some authors (see, e.g., \cite{AM,BA,TA}) use the term ``$KB$-operator" for $(d)$ of Definition \ref{order-to-topology} which is slightly confusing
because it says nothing about existence of limits of topologically bounded increasing nets and all continuous finite-rank operators in every Banach lattice
satisfy $(d)$. Each order bounded operator from a Banach lattice to a $KB$-space is quasi $KB$. However, if we take $X=(c_{00},\|\cdot\|_l)$ and $Y=(c_{00},\|\cdot\|_p)$ 
with $l,p\in [1,\infty]$, the identity operator $I$ is quasi $KB$ iff $l\le p<\infty$. 
\item[$e)$] \
It is clear that every compact operator $T$ from a Banach lattice $X$ to a Banach space $Y$ is $o\tau$-compact. 
However, compact operator need not to be Lebesgue (cf. Example \ref{c_w(R)}). In particular, $o\tau$-compact operators are not necessarily  $o\tau$-continuous.
\end{enumerate}}
\end{rem}

\begin{exam}\label{c_0(R)} 
Let $(c(\mathbb{R}),\|\cdot\|_\infty)$ be the Banach lattice of all $\mathbb{R}$-valued functions on $\mathbb{R}$ 
such that for every $f\in c(\mathbb{R})$ there exists $a_f\in\mathbb{R}$ for which the set $\{r\in\mathbb{R}: |f(r)-a_f|\ge\varepsilon\}$ 
is finite for each $\varepsilon>0$. Then the identity operator $I$ in $c(\mathbb{R})$ is $\sigma{}o\tau$-continuous and quasi $\sigma$-Lebesgue yet neither Lebesgue nor quasi Lebesgue.
\end{exam}

\begin{exam}\label{c_00(H)} 
Recall that, for a nonempty set $H$, the vector space $c_{00}(H)$ of all finitely supported $\mathbb{R}$-valued functions on $H$ is a Dedekind complete vector lattice.
Furthermore, any vector space $X$ is linearly isomorphic to $c_{00}(H)$, where $H$ is a Hamel basis for $X$. As each order interval of $c_{00}(H)$ lies in a
finite-dimensional subspace of $c_{00}(H)$, every operator from $c_{00}(H)$ to any topological vector space $(Y,\tau)$ is $o\tau$-continuous,
$o\tau$-bounded, and $o\tau$-compact. On the other hand, the identity operator in $(c_{00},\|\cdot\|_\infty)$ is neither quasi $\sigma$-$KB$ nor quasi $\sigma$-Levi. 
\end{exam}

\begin{rem}\label{tau-Cauchy nets} 
{\em
It is well known that, if a net $y_\alpha$ of a topological vector space $(Y,\tau)$ is not $\tau$-Cauchy, then there exist $U\in\tau(0)$ and an increasing sequence $\alpha_n$
such that $y_{\alpha_{n+1}}-y_{\alpha_n}\not\in U$ for each $n$ (see, e.g., \cite[Lem.2.5]{AB1}).}   
\end{rem}

\begin{prop}\label{quasi-KB-vs-sigma-quasi-KB}
An operator $T$ from a locally solid lattice $(X,\varsigma)$ to a topological vector space $(Y,\tau)$ is quasi $KB$ iff $T$ is quasi $\sigma$-$KB$.
\end{prop}

\begin{proof}
The necessity is trivial. For the sufficiency, suppose $T$ is not quasi $KB$. 
Then there exists a $\varsigma$-bounded increasing net $x_\alpha$ in $X_+$ 
such that $Tx_\alpha$ is not $\tau$-Cauchy in $Y$. It follows from Remark \ref{tau-Cauchy nets}
that for some increasing sequence $\alpha_n$ and some $U\in\tau(0)$ 
$$
    Tx_{\alpha_{n+1}}-Tx_{\alpha_n}\not\in U \ \ \ (\forall n\in\mathbb{N}) 
    \eqno(1)
$$
Since the sequence $x_{\alpha_n}$ is increasing and $\varsigma$-bounded, the condition $(1)$ implies that $T$ is not quasi $\sigma$-$KB$.
\end{proof}

\begin{cor}\label{s-complete is sigma-KB is KB}
Every $\varsigma$-complete $\sigma$-$KB$ lattice $(X,\varsigma)$ is a $KB$ lattice.
\end{cor}

\begin{proof}
The identity operator $I$ on $X$ is $\sigma$-$KB$ and hence quasi $\sigma$-$KB$.
By Proposition \ref{quasi-KB-vs-sigma-quasi-KB}, $I$ is quasi $KB$.
Then every $\varsigma$-bounded increasing net in $X_+$ is $\varsigma$-Cauchy, and hence is $\varsigma$-convergent.
\end{proof}
\begin{rem}\label{sequential}{\em
The following fact (cf., e.g., \cite[Prop.1.1]{AEG}) is well known: 
\begin{enumerate}
\item[$a)$] \
$\rho(y_\alpha, y)\to 0$ in a metric space $(Y,\rho)$ iff, for every subnet $y_{\alpha_\beta}$ of the net $y_\alpha$, 
there exists a (not necessary increasing) sequence $\beta_k$ of indices with $\rho(y_{\alpha_{\beta_k}}, y)\to 0$.
\end{enumerate}
Application of $a)$ to Cauchy nets/sequences in $(Y,\rho)$ gives that:
\begin{enumerate}
\item[$b)$] \
a net/sequence $y_\alpha$ of a metric space $(Y,\rho)$ is Cauchy iff, for every subnet/subsequence $y_{\alpha_\beta}$ 
of $y_\alpha$, there exists a sequence $\beta_k$ of indices such that the sequence $y_{\alpha_{\beta_k}}$ is $\rho$-Cauchy.
\end{enumerate}
Application of $b)$ and Proposition \ref{quasi-KB-vs-sigma-quasi-KB} gives that, for an operator $T$ 
from a locally solid lattice $(X,\varsigma)$ to a metric vector space $(Y,\rho)$, the following are equivalent:
\begin{enumerate}
\item[$(i)$] \
$T$ is quasi $KB$;
\item[$(ii)$] \
Every $\varsigma$-bounded increasing sequence $x_n$ in $X_+$ has a subsequence $x_{n_k}$ such that $Tx_{n_k}$ is $\rho$-Cauchy in $Y$. 
\end{enumerate}
Another application of $b)$ to an operator $T$ from a locally solid lattice $(X,\varsigma)$ to a metric vector space $(Y,\rho)$ gives the equivalence of the following conditions:
\begin{enumerate}
\item[$(i)'$] \
$T$ is $KB$;
\item[$(ii)'$] \
For any $\varsigma$-bounded increasing net $x_\alpha$ in $X_+$, there exist an element $x\in X$ and a sequence $\beta_k$ of indices 
such that $\rho(Tx_{\alpha_{\beta_k}},Tx)\to 0$.
\end{enumerate} }
\end{rem}

In many cases, like for (quasi) $KB$ or Levi operators, the lattice structure in the domain/range of operator can be relaxed to the ordered space structure \cite{AM,AGG,EG,EM2},
or combined with the lattice-norm structure \cite{AP,AEEM1,AEEM2,AGG}. Such generalizations are not included in the present paper. In Section 2 we investigate operators whose 
domains and/or ranges are locally solid lattices. Section 3 is devoted to operators between Banach lattices.

\section{Operators between locally solid lattices}

In this section, we study mostly operators whose domains and/or ranges are locally solid lattices. 
Observe first that the sets $L_{Leb}(X,Y)$, $L_{o\tau}(X,Y)$, $L_{o\tau{}b}(X,Y)$, and $L_{o\tau{}c}(X,Y)$ of 
Le\-bes\-gue, $o\tau$-continuous, $o\tau$-bounded, and $o\tau$-compact operators respectively from a ve\-ctor lattice $X$ 
to a topological vector space $(Y,\tau)$ are vector spaces and:
\begin{enumerate}
\item[$(*)$] \  $L_{o\tau}(X,Y)\subseteq L_{Leb}(X,Y)$;  
\item[$(**)$] \  $L_{o\tau c}(X,Y)\subseteq L_{o\tau b}(X,Y)$ since every totally bounded subset in $Y$ is bounded.  
\end{enumerate} 

\begin{theorem}\label{Thm.5.10 of AB2}
Let $X$ be a vector lattice and $(Y,\tau)$ be a locally convex Lebesgue lattice which is either Dede\-kind complete or $\tau$-complete.
Then $L_{o\tau c}(X,Y)\bigcap L_r(X,Y)$ is a band of the lattice $L_r(X,Y)$ of all regular operators from $X$ to $Y$. 
\end{theorem}

\begin{proof}
Since every Lebesgue lattice is pre-Lebesgue by \cite[Thm.3.23]{AB1}, and every topologically complete pre-Lebesgue lattice is Dedekind complete by
Proposition \ref{tau compl pre Leb is Ded compl}, the lattice $(Y,\tau)$ is Dede\-kind complete in either case. By \cite[Thm.5.10]{AB2}, 
for each $x\in X_+$, the set 
$$
   C(x)=\{T\in L_b(X,Y): T[0,x] \ \text{is}\  \tau\text{-totally}\ \text{bounded}\}
$$ 
is a band of the Dedekind complete lattice $L_b(X,Y)$ of all order bounded operators from $X$ to $Y$. 
Since $L_r(X,Y)=L_b(X,Y)$, the set $L_{o\tau{}c}(X,Y)\bigcap L_r(X,Y)=\bigcap_{x\in X_+}C(x)$ is also a band of $L_r(X,Y)$ as desired.
\end{proof}

The following two propositions might be known. We include their elementary proofs as we did not find appropriate references.    

\begin{prop}\label{regular are otau-bounded}
Each regular operator $T$ from a vector lattice $X$ to a locally solid lattice $(Y,\tau)$ is $o\tau$-bounded. 
\end{prop}

\begin{proof}
Without lost of generality, assume $T\ge 0$. Let $x\in X_+$, $U\in\tau(0)$, and $V$ be a solid $\tau$-neighborhood of zero of $Y$ with $V\subseteq U$. 
Then $Tx\in\lambda V$ for all $\lambda\ge\lambda_0$ for some $\lambda_0>0$ and hence $T[0,x]\subseteq [0,Tx]\subseteq\lambda V\subseteq \lambda U$ 
for all $\lambda\ge\lambda_0$. Since $U\in\tau(0)$ was arbitrary, $T$ is $o\tau$-bounded.
\end{proof}

\begin{prop}\label{Lebesgue vs weak Lebesgue}
An order continuous positive operator $T$ from a vector lattice $X$ to a locally convex-solid lattice $(Y,\tau)$ is Lebesgue iff $T$ is weakly Lebesgue.
\end{prop}

\begin{proof}
The necessity is trivial. For the sufficiency, assume that $T$ is weakly Lebesgue (i.e., $T:X\to Y$ is Lebesgue with respect to the weak topology $\sigma(Y,Y')$ on $Y$).
Let $x_\alpha\downarrow 0$ in $X$. Then $Tx_\alpha\downarrow 0$ in $Y$. Since $T$ is $\sigma(Y,Y')$-Lebesgue, $Tx_\alpha\xrightarrow[]{\sigma(Y,Y')}0$. 
By using Dini-type arguments (cf. \cite[Thm.3.52]{AB2}), we conclude $Tx_\alpha\convtau 0$, as desired.
\end{proof}

\noindent
We do not know any example of an order continuous $o\tau$-bounded weakly Lebesgue operator from a vector lattice to a locally solid lattice that is not Lebesgue.

\begin{prop}\label{Lebesgue vs weak weak cont}
A weakly continuous positive operator $T$ from a Lebesgue $($$\sigma$-Lebesgue$)$ lattice $(X, \varsigma )$ to 
a locally solid lattice $(Y,\tau)$ is Lebesgue $($$\sigma$-Lebesgue$)$.
\end{prop}

\begin{proof} 
As arguments are similar, we restrict ourselves to the Lebesgue case. Let $x_\alpha\downarrow 0$ in $X$.
Since $(X, \varsigma )$ is Lebesgue, $x_\alpha\convvars 0$ and hence $x_\alpha\xrightarrow[]{\sigma(X,X')}0$. The weak continuity of $T$
implies $Tx_\alpha\xrightarrow[]{\sigma(Y,Y')}0$. Since $Tx_\alpha\downarrow 0$ in $Y$, it follows $Tx_\alpha\convtau 0$, as 
in the proof of Proposition \ref{Lebesgue vs weak Lebesgue}.
\end{proof}

\begin{theorem}\label{Thm.3.23 from AB2}
Each positive Lebesgue operator $T$ from a vector lattice $X$ to a locally solid lattice $(Y,\tau)$ is quasi Lebesgue.
\end{theorem}

\begin{proof}
Let $x_\alpha\uparrow\le x\in X$, so $u_\alpha:=x-x_\alpha\downarrow\ge 0$.
Therefore $(u_\alpha-z)_{\alpha;z\in L}\downarrow 0$, where
$L=\{z\in X: (\forall\alpha\in A)[z\le u_\alpha]\}$ is the upward directed set of lower bounds of the net $u_\alpha$
(cf. \cite[Thm.1.2]{Em}). Since $T$ is Lebesgue and $T\ge 0$, 
$$
  (Tu_\alpha-Tz)_{\alpha;z\in L}\convtau 0 \ \ \ \& \ \ \ (Tu_\alpha-Tz)_{\alpha;z\in L}\downarrow\ge 0.
  \eqno(2)
$$
It follows from $(2)$, that $(Tu_\alpha-Tz)_{\alpha;z\in L}\downarrow 0$. Then $Tu_\alpha$ and hence $Tx_\alpha$ is $\tau$-Cauchy, as desired.
\end{proof}

\noindent
We do not know any example of an $o\tau$-bounded Lebesgue operator which is not quasi Lebesgue. In contrast to Theorem 3.24 of \cite{AB1}, 
the converse of Theorem \ref{Thm.3.23 from AB2} is false even if $Y=\mathbb{R}$ due to Example \ref{c_w(R)}.
 
Recall that an Archimedean vector lattice $X$ is called {\em laterally} ({\em $\sigma$-}) {\em complete}
whenever every (countable) subset of pairwise disjoint vectors of $X_+$ has a supremum.
Every laterally complete vector lattice $X$ contains a weak order unit and every band of such an $X$ is a principal band (cf. \cite[Thm.7.2]{AB1}). 
Furthermore, by the Veksler~-- Geiler theorem  (cf. \cite[Thm.7.4]{AB1}), if a vector lattice $X$ is laterally ($\sigma$-) complete, then 
$X$ satisfies the (principal projection) projection property. 

A vector lattice that is both laterally and Dedekind ($\sigma$-) complete is referred 
to as {\em universally} ($\sigma$-) {\em complete}. It follows from \cite[Thm.7.4]{AB1} that a vector lattice $X$ is universally complete 
iff $X$ is Dedekind $\sigma$-complete and laterally complete iff $X$ is uniformly complete and laterally complete (cf. \cite[Thm.7.5]{AB1}). 
Similarly, a laterally $\sigma$-complete vector lattice $X$ is Dedekind $\sigma$-complete iff $X$ is uniformly complete.

A universal completion of a vector lattice $X$ is a laterally and De\-de\-kind complete vector lattice $X^u$ which contains $X$ as an order dense sublattice. 
Every vector lattice has  unique universal completion (cf. \cite[Thm.7.21]{AB1}).

A laterally complete vector lattice $X$ is discrete iff $X$ is lattice isomorphic to $\mathbb{R}^S$ for some nonempty set $S$ iff $X$ admits 
a Hausdorff locally convex-solid Lebesgue topology iff the space $X_n^\sim$ of order continuous functionals on $X$ separates $X$ (cf. \cite[Thm.7.48]{AB1}); 
in each of these cases $X = X^u$.

\begin{definition}\label{tau-laterally-complete}
{\em
A topological vector lattice $(Y,\tau)$ is called {\em $\tau$-laterally} ({\em $\sigma$-}) {\em complete} 
whenever every $\tau$-bounded (countable) subset of pairwise disjoint vectors of $X_+$ has a supremum.
}
\end{definition}

Every laterally ($\sigma$-) complete topological vector lattice $(Y,\tau)$ is $\tau$-laterally ($\sigma$-) complete. Every Dedekind complete $AM$-space 
$X$ with an order unit  is $\tau$-laterally complete with respect to the norm.

\begin{exam}\label{not sigma Lebesgue tau-laterally}
{\em
Let $X$ be a vector lattice of real functions on $\mathbb{R}$ such that each $f\in X$ may differ from a constant say $a_f$ on a countable subset of $\mathbb{R}$
and $f-a_f\mathbb{I}_\mathbb{R}\in\ell_1(\mathbb{R})$ for each $f\in X$. 
\begin{enumerate}
\item[$(i)$] \ 
The vector lattice $X$ is not Dedekind $\sigma$-complete, as $f_n:=\mathbb{I}_{\mathbb{R}\setminus\{1,2,...,n\}}\downarrow \ge 0$ yet 
$\inf\limits_{n \in \mathbb{N}}f_n$ does not exist in $X$.
\item[$(ii)$] \ 
The vector lattice $X$ is a Banach space with respect to the norm $\|f\|:=|a_f|+\|f-a_f\mathbb{I}_\mathbb{R}\|_1$.
\item[$(iii)$] \ 
The vector lattice $X$ is not $\tau$-laterally $\sigma$-complete with respect to the norm topology on $X$. Indeed,
the norm bounded countable set of pairwise disjoint orths $e_n=\mathbb{I}_{\{n\}}$ of $X_+$ has no supremum in $X$.
\item[$(iv)$] \ 
The identity operator $I$ on $X$ is not $\sigma$-Lebesgue, as\\
$\frac{1}{n}\mathbb{I}_{\mathbb{R}\setminus\{1,2,...,n\}}\downarrow 0$ in $X$ yet
$\|\frac{1}{n}\mathbb{I}_{\mathbb{R}\setminus\{1,2,...,n\}}\|=\frac{n+1}{n}\not\to 0$.
\end{enumerate}
}
\end{exam}

\begin{prop}\label{straightforward}
Any Levi $($$\sigma$-Levi$)$ lattice is $\tau$-la\-te\-rally $($$\sigma$-$)$ complete and Dedekind $($$\sigma$-$)$ complete.
\end{prop}

\begin{proof}
We consider the Levi case only, because the $\sigma$-Levi case is similar. The  Dedekind completeness of $X$ is trivial. 
Let $D$ be a $\tau$-bounded subset of pairwise disjoint positive vectors of a Levi lattice $(Y,\tau)$. Then the collection $D^\vee$ 
of suprema of finite subsets of $D$ forms an increasing $\tau$-bounded net indexed by the set ${\cal P}_{fin}(D)$ of all finite subsets of $D$ 
directed by inclusion.  Since $(Y,\tau)$ is Levi,  $D^\vee \uparrow d$ for some $d \in Y$. It follows from $\sup D = \sup D^\vee =d$ that 
$(Y,\tau)$ is  $\tau$-la\-te\-rally complete.
\end{proof}

\begin{exam}\label{s}
The locally solid lattice $(\mathbb{R}^S,\tau)$, where $\tau$ is the product topology on the vector lattice $\mathbb{R}^S$ of real functions 
on a set $S$ is locally convex,  Lebesgue, Levi, and universally complete.
\end{exam}

\begin{prop}\label{Thm.7.8. from AB1}
Each regular operator $T$ from a laterally $\sigma$-complete vector lattice $X$ to a $\sigma$-Lebesgue lattice $(Y,\tau)$ is $\sigma$--Lebesgue.
\end{prop}

\begin{proof}
Without lost of generality, we can suppose $T\ge 0$. Let $x_n\downarrow 0$ in $X$. By Theorem 7.8 of \cite{AB1}, $T$ is $\sigma$-order continuous, 
and hence $Tx_n\downarrow 0$ in $Y$. Since $(Y,\tau)$ is $\sigma$--Lebesgue, $Tx_n\convtau 0$.
\end{proof}

\noindent
We do not know whether or not every positive operator $T$ from a laterally complete vector lattice $X$ to a Lebesgue lattice $(Y,\tau)$ is Lebesgue.

Each laterally $\sigma$-complete locally solid lattice is $\sigma$-Le\-bes\-gue and pre-Le\-bes\-gue \cite[Thm.7.49]{AB1}.

\begin{prop}\label{varsigma1}
Let $T$ be a continuous operator from a laterally $\sigma$-complete locally solid lattice $(X,\varsigma)$ to a topological vector space $(Y, \tau)$.
Then  $T$ is $\sigma$-Lebesgue in each of the following cases:
\begin{enumerate}
\item[$(i)$] \  
$X$ is discrete.
\item[$(ii)$] \  
$(X, \varsigma)$ is metrizable.
\end{enumerate}
\end{prop}

\begin{proof}
Let $x_n\downarrow 0$ in $X$. In case $(i)$, $(X,\varsigma)$ is $\sigma$-Lebesgue by \cite[Thm.7.49]{AB1}. The result follows because $T$ is $\varsigma\tau$-continuous.
In case $(ii)$, the result follows from  $(i)$ as every infinite dimensional laterally $\sigma$-complete metrizable locally solid lattice is
lattice isomorphic to $\mathbb{R}^\mathbb{N}$.
\end{proof}

Since $(X,u\varsigma)$ is pre-Lebesgue iff $(X,\varsigma)$ is pre-Lebesgue \cite[Cor.4.3]{Tay2}, the statement of Proposition \ref{varsigma1}
is also true when $(X,\varsigma)$ is replaced by $(X,u\varsigma)$. 

\begin{rem} 
{\em
Let $T$ be a positive operator from a locally solid lattice $(X,\varsigma)$ to a laterally $\sigma$-complete locally solid lattice $(Y, \tau)$ and let 
$0\le x_\alpha\uparrow \le x$ in $X$. Since $(Y, \tau)$ is pre-Lebesgue by \cite[7.49]{AB1}, it follows from $0\le Tx_\alpha\uparrow \le Tx$ that
$Tx_\alpha$ is $\tau$-Cauchy in $Y$ and hence $T$ is quasi Lebesgue.
}
\end{rem}

\begin{rem}
{\em
Let $T$ be a continuous operator from a laterally $\sigma$-complete locally solid lattice $(X,\varsigma)$ to a locally solid lattice $(Y,\tau)$.
\begin{enumerate}
\item[$a)$] \
Suppose $T\ge 0$ and $0\le x_\alpha\uparrow\le x$ in $X$. Since $(X,\varsigma)$ is pre-Lebesgue by \cite[7.49]{AB1} then $x_\alpha$ is $\varsigma$-Cauchy in $X$ and hence 
$Tx_\alpha$ is $\tau$-Cauchy in $Y$ meaning that $T$ is quasi Lebesgue.
\item[$b)$] \
Suppose $(X,\varsigma)$ is Fatou, and $x_\alpha\downarrow 0$ in $X$. Then by \cite[7.50]{AB1} 
$(X,\varsigma)$ is Lebesgue, so $x_\alpha\convvars 0$ and hence $Tx_\alpha\convtau 0$ meaning that $T$ is Lebesgue.
\item[$c)$] \
Suppose $\varsigma$ is a metrizable locally solid  topology and $x_\alpha\downarrow 0$ in $X$.  
Then $(X,\varsigma)$ is Lebesgue by \cite[Thm.7.55]{AB1}. Thus $x_\alpha\convvars 0$ and hence $Tx_\alpha\convtau 0$ meaning that $T$ is Lebesgue.
\item[$d)$] \
If $(X,\varsigma)$ is a Fr{\'e}chet lattice then $\varsigma$ is uniqu\-ely defined and every regular $T:(X, \varsigma)\to(Y, \tau)$ is continuous 
(cf. \cite[Thm.5.19 and Thm.5.21]{AB1}). In particular, every regular operator from a laterally $\sigma$-complete Fr{\'e}chet lattice $X(\varsigma)$  
to a locally solid lattice $(Y,\tau)$ is Lebesgue.
\end{enumerate}
}
\end{rem} 

\begin{prop}\label{cont are otau-bounded}
Any continuous operator $T$ from a locally solid lattice $(X,\varsigma)$ to a topological vector space $(Y,\tau)$ is $o\tau$-boun\-ded. 
\end{prop}

\begin{proof}
Let $x\in X_+$, $U\in\tau(0)$, and $V$ be a solid $\varsigma$-neighbor\-hood of zero of $X$ with $V\subseteq T^{-1}U$. 
Then there exists $\lambda_0>0$ such that $x\in\lambda V$ for all $\lambda\ge\lambda_0$ and hence $[0,x]\subseteq\lambda V\subseteq\lambda T^{-1}(U)$
for all $\lambda\ge\lambda_0$. Thus $T[0,x]\subseteq\lambda U$ for all $\lambda\ge\lambda_0$, which shows that $T$ is $o\tau$-bounded.
\end{proof}

\noindent
It is worth mentioning here that an $o\tau$-boun\-ded operator $T$ from Dedekind $\sigma$-complete vector lattice $X$ to a normed lattice $Y$ 
is $\sigma$-Lebesgue iff $T$ is order-weakly compact (cf. \cite[Cor.1]{JAM}). The following lemma can be considered as an extension of \cite[Lem.1]{JAM} 
to locally solid lattices.

\begin{lem}\label{PC1}
A positive operator $T$ from a vector  lattice $X$ to a locally solid lattice $(Y,\tau)$ is Lebesgue/$\sigma$-Lebesgue iff $T$ is $o\tau$-continuous/$\sigma{}o\tau$-continuous.
\end{lem}

\begin{proof}
As the $\sigma$-Lebesgue case is similar, we consider only Lebesgue operators. The sufficiency is routine.
For the necessity, assume $T$ is Lebes\-gue and $x_\alpha\convo 0$ in $X$. Take a net $z_\beta$ in $X$ with $z_\beta\downarrow 0$
such that, for each $\beta$, there exists $\alpha_\beta$ with $|x_\alpha|\le z_\beta$ for $\alpha\ge\alpha_\beta$. It follows from $T\ge 0$ 
that $|Tx_\alpha|\le T|x_\alpha|\le Tz_\beta$ for $\alpha\geq\alpha_\beta$. As $T$ is Lebes\-gue, $Tz_\beta\convtau 0$, which implies $Tx_\alpha\convtau 0$,
because the topology $\tau$ is locally solid.
\end{proof}

\begin{cor}\label{PC2}
An order bounded operator $T$ from a vector lattice to a Dedekind complete locally solid lattice is Lebesgue/$\sigma$-Lebesgue iff $T$ is 
$o\tau$-continuous/$\sigma{}o\tau$-conti\-n\-u\-ous.
\end{cor}

Now we investigate some properties of adjoint operators. Recall that, for a locally solid lattice $(X,\varsigma)$, 
the {\em absolute weak topology} $|\sigma|(X^\sim,X)$ on $X^\sim$ is the locally convex-solid topology generated by the collection 
of Riesz seminorms $\{\rho_x(f)=|f|(|x|): x\in X\}$.  The locally solid lattice $(X^\sim,|\sigma|(X^\sim,X))$ is Lebesgue, Levi, 
and Fatou \cite[Prop.81C]{Fre}. We begin with the following technical lemma. 

\begin{lem}\label{dual is cont}
Let $T:(X,\varsigma)\to(Y,\tau)$ be an order bounded continuous operator from a locally solid lattice $(X,\varsigma)$ to a Dedekind complete locally solid lattice $(Y,\tau)$. 
Then the topological adjoint $T':(Y',|\sigma|(Y^\sim,Y))\to(X',|\sigma|(X^\sim,X))$ is also continuous. 
\end{lem}

\begin{proof}
Indeed, for $f_\alpha\xrightarrow{|\sigma|(Y^\sim,Y)}0$ in $Y^\sim$, it follows from
$$
   \rho_{x}(T'f_\alpha)=\langle |T'f_\alpha|, |x|\rangle \le \langle |T'||f_\alpha|, |x|\rangle\le\langle |T|' |f_\alpha|, |x|\rangle = 
$$
$$
   \langle|f_\alpha|, |T||x|\rangle=|f_\alpha|(|T||x|)=\rho_{|T||x|}(f_\alpha)\to 0 
  \eqno(3) 
$$
that $T'f_\alpha\xrightarrow{|\sigma|(X^\sim,X)}0$. The second inequality in $(3)$ follows from the existence of the modulus of $T$ and from
the observation $\pm T\le|T|\Rightarrow \pm T'\le|T|'$ (cf. \cite[p.67]{AB2}) which implies $|T'|\le|T|'$.
\end{proof}

\begin{theorem}
Let $T: X \to Y$ be an order bounded operator from a vector lattice $X$ to a  Dedekind complete vector lattice $Y$, and 
$T':(Y^\sim,|\sigma|(Y^\sim,Y))\to(X^\sim,|\sigma|(X^\sim,X))$ the correspondent topological adjoint of $T$. Then $T'$ is $o\tau$-bounded, $o\tau$-continuous, Levi, and $KB$.
\end{theorem}

\begin{proof}
The $o\tau$-boundedness of $T'$ follows from Lemma \ref{dual is cont} by Proposition \ref{cont are otau-bounded}.
Since $T$ is regular, without lost of generality, we can suppose $T\ge 0$ (and hence $T'\ge 0$).

Let $f_\alpha\convo 0$ in $Y^\sim$. As $T'$ is order continuous, $T'f_\alpha\convo 0$ in $X^\sim$, and since  $(X^\sim,|\sigma|(X^\sim,X))$ is a Lebesgue lattice, 
we obtain $T'f_\alpha\xrightarrow[]{|\sigma|(X^\sim,X)}0$ and hence $T'$ is $o\tau$-continuous. 

Let $f_\alpha$ be a positive $|\sigma|(Y^\sim,Y)$-bounded increasing net in $Y^\sim$.
Since $(Y^\sim,|\sigma|(Y^\sim,Y))$ is Levi, $f_\alpha\uparrow f$ for some $f\in Y^\sim$. 
As $T'$ is order continuous, $T'f_\alpha\uparrow T'f$ in $X^\sim$, in particular $T'$ is Levi. 
Since the locally solid lattice $(X^\sim,|\sigma|(X^\sim,X))$ is Lebesgue, $T'f_\alpha\xrightarrow[]{|\sigma|(X^\sim,X)}T'f$ and hence 
$T': (Y^\sim,|\sigma|(Y^\sim,Y)) \to (X^\sim,|\sigma|(X^\sim,X))$ is $KB$. 
\end{proof}

The natural embedding $J$ of a vector lattice $X$ to $(X_n^\sim)_n^\sim$ is defined by $(Jx)(f) = f(x)$ for all $f\in X_n^\sim$, $x\in X$.
The mapping $J$ is an order continuous lattice homomorphism (cf. e.g. \cite[Thm.1.70]{AB2}).
By the Nakano theorem (cf. \cite[Thm.1.71]{AB2}), $J$ is one-to-one and onto (i.e. $X$ is {\em perfect})
iff $X_n^\sim$ separates $X$ and whenever a net $x_\alpha$ of $X_+$ satisfies $x_\alpha\uparrow$ and $\sup f(x_\alpha)<\infty$
for each $0\le f\in X_n^\sim$, then there exists some $x\in X$ satisfying $x_\alpha\uparrow x$ in $X$.

\begin{lem}\label{sveta5}
Let $T$ be an order bounded operator from a vector lattice $X$ to a vector lattice $Y$.
The second order adjoint $T^{\sim\sim}$, while restricted to $(X_n^\sim)^\sim$, 
$$
    ((X_n^\sim)^\sim, |\sigma|((X_n^\sim)^\sim,X_n^\sim))\xrightarrow{T^{\sim\sim}}((Y^\sim)^\sim, |\sigma|((Y_n^\sim)^\sim,Y_n^\sim))
$$ 
is continuous. 
\end{lem}

\begin{proof}
Let $(X_n^\sim)^\sim\ni  x_\alpha \xrightarrow{|\sigma|((X_n^\sim)^\sim,X_n^\sim)} 0$. 
By use of the fact $|T^{\sim\sim}|\le|T^\sim|^\sim$ it follows
$$ 
   \langle |T^{\sim\sim}x_\alpha|, |y|\rangle \le \langle|T^\sim|^\sim|x_\alpha|, |y|\rangle \le \langle |x_\alpha|,|T^\sim||y|\rangle
   \eqno(4)
$$
for all $y\in Y_n^\sim$.  Since  $x_\alpha \xrightarrow{|\sigma|((X_n^\sim)^\sim, X_n^\sim)} 0$, it follows from $(4)$
that $T^{\sim\sim}x_\alpha \xrightarrow{|\sigma|((Y_n^\sim)^\sim,Y_n^\sim)}0$ and hence the operator
$T^{\sim\sim}$ is $|\sigma|((X_n^\sim)^\sim,X_n^\sim)$ to $|\sigma|((Y_n^\sim)^\sim,Y_n^\sim)$ continuous. 
\end{proof}

\begin{theorem}\label{sveta6}
Let $T$ be an order bounded operator from a vector lattice $X$ to a vector lattice $Y$. Then
$$
  ((X_n^\sim)_n^\sim,|\sigma|((X_n^\sim)_n^\sim,X_n^\sim))\xrightarrow{T^{\sim\sim}}((Y_n^\sim)_n^\sim, |\sigma|((Y_n^\sim)_n^\sim,Y_n^\sim)).
$$ 
is a continuous, Lebesgue, Levi, and $KB$ operator.
\end{theorem}

\begin{proof}
The continuity follows from Lemma \ref{sveta5} by restricting of the second order adjoint $T^{\sim\sim}$ to $(X_n^\sim)_n^\sim$.

Let $x_\alpha\downarrow 0$ in  $(X_n^\sim)_n^\sim$. As $T^{\sim\sim}$ is order continuous, $T^{\sim\sim}x_\alpha\downarrow 0$
in  $(Y_n^\sim)_n^\sim$. Since $((Y_n^\sim)_n^\sim,|\sigma|((Y_n^\sim)_n^\sim,Y_n^\sim))$ is Lebesgue,
$T^{\sim\sim}x_\alpha\xrightarrow{|\sigma|((Y_n^\sim)_n^\sim,Y_n^\sim)} 0$ showing that the operator is Lebesgue. 

Let $x_\alpha$ be a positive increasing $|\sigma|((X_n^\sim)_n^\sim, X_n^\sim)$-bounded net in $(X_n^\sim)_n^\sim$.
Since $((X_n^\sim)_n^\sim,|\sigma|((X_n^\sim)_n^\sim,X_n^\sim))$ is Levi,  the net $x_\alpha$ has a supremum in $(X_n^\sim)_n^\sim$, say $x$. 
Thus $x_\alpha\uparrow x$ and $T^{\sim\sim}x_\alpha\uparrow T^{\sim\sim}x$ as $T^{\sim\sim}$ is order continuous. 
Thus the operator is Levi. 

Since  $((X_n^\sim)_n^\sim,|\sigma|((X_n^\sim)_n^\sim,X_n^\sim))$ is Levi,  
$x_\alpha\uparrow x$ in $(X_n^\sim)_n^\sim$. Then $T^{\sim\sim} x_\alpha\convo {T^\sim}^\sim x$ by order continuity of the operator $T^{\sim\sim}$. 
As  $((Y_n^\sim)_n^\sim,|\sigma|((Y_n^\sim)_n^\sim,Y_n^\sim))$ is a Lebesgue lattice, 
$T^{\sim\sim} x_\alpha\xrightarrow{|\sigma|((Y_n^\sim)_n^\sim,Y_n^\sim)}T^{\sim\sim}x$ and  the operator is $KB$.
\end{proof}

A vector sublattice $Z$ of a vector lattice $X$ is called {\em regular}, if the embedding of $Z$ into $X$ preserves arbitrary suprema and infima. 
Order ideals and order dense vector sublattices of a vector lattice $X$ are regular \cite[Thm.1.23]{AB1}. 

\begin{cor}\label{sveta7}
Let $X^{\sim}=X_n^{\sim}$, and $Y$ be a vector lattice.
Then each order bounded operator $T:X\to(Y, |\sigma|(Y,Y_n^\sim))$ is Lebes\-gue. 
\end{cor}

\begin{proof}
Let $x_\alpha\downarrow 0$ in $X$. Since the natural embedding $J:X\to X^{\sim\sim}$ is one-to-one and $J(X)$ is 
a regular sublattice of $X^{\sim\sim}$ (cf. \cite[Thm.1.67]{AB1}), then $T=T^{\sim\sim}\circ J$ and
$Jx_\alpha\downarrow 0$ in $X^{\sim\sim}$. Thus $Tx_\alpha=T^{\sim\sim}(Jx_\alpha)\xrightarrow{|\sigma|((X_n^\sim)_n^\sim,X_n^\sim)}0$
since $T^{\sim\sim}$ is Lebesgue by Theorem \ref{sveta6}. 
\end{proof}

\medskip
We also mention the following question. Let $T$ be Lebes\-gue, $o\tau$-continuous, $o\tau$-bounded, $o\tau$-compact, $KB$, or Levi. Does $T^{\sim\sim}$ satisfy the same property?
The answer is negative,  when $T$ maps a Banach lattice $X$ to a Banach lattice $Y$, in the following cases: 
\begin{enumerate}
\item[$(*)$] \  
for Lebesgue and for $o\tau$-continuous operators, since the identity $I:c_0\to c_0$ is $o\tau$-continuous yet its second order adjoint $I:\ell_\infty\to\ell_\infty$ is not even Lebesgue; 
\item[$(**)$] \  
for $o\tau$-compact operators, since the natural embedding $J:c_0\to\ell_\infty$ is $o\tau$-compact but $J^{\sim\sim}:\ell_\infty\to\ell^{**}_\infty$ is not $o\tau$-compact.
\end{enumerate} 

Next we discuss the domination properties of positive operators. Let $T$ and $S$ be positive operators between vector lattices $X$ and $Y$ satisfying $0\le S\le T$. 

\begin{quest}\label{dominated Lebesgue}
{\em 
When does the assumption that $T$ is Lebes\-gue, $o\tau$-continuous, $o\tau$-bounded, $o\tau$-compact, $KB$, or Levi imply that $S$ has the same property?
}
\end{quest}

\noindent
Question \ref{dominated Lebesgue} has positive answers in the following cases: 
\begin{enumerate}
\item[$<a>$] \
Trivially, for Lebesgue; $\sigma$-Lebesgue; and $o\tau$-bounded operators from a vector lattice to a locally solid lattice.
\item[$<b>$] \
For $o\tau$-compact operators from a vector lattice $X$ to a locally convex Lebesgue lattice $(Y,\tau)$ which is either Dede\-kind complete or $\tau$-complete by Theorem \ref{Thm.5.10 of AB2}.
\item[$<c>$] \
For $o|\sigma|(Y,Y')$-compact operators from a Banach lattice $X$ to a Banach lattice $Y$ \cite[Thm.5.11]{AB2}.
\item[$<d>$] \
For $KB$/$\sigma$-$KB$ lattice homomorphisms between locally solid lattices by Corollary \ref{KB lattice homomorphism dominated property}.
\item[$<e>$] \
For quasi $KB$ operators between locally solid lattices by Theorem \ref{quasi $KB$ dominated property}.
\item[$<f>$] \
For quasi Levi operators from a locally solid lattice to a vector lattice by Theorem \ref{quasi Levi dominated property}.
\end{enumerate}

We remind that the modulus $|T|$ of an order bounded disjointness preserving operator $T$ between vector lattices $X$ and $Y$
exists and satisfies $|T||x|=|T|x||=|Tx|$ for all $x\in X$ \cite[Thm.2.40]{AB2};
there exist lattice homomorphisms $R_1,R_2:X\to Y$ with $T = R_1-R_2$ (cf. \cite[Exer.1, p.130]{AB2}. 

\begin{theorem}\label{KB disj pres dominated property}
Let $T$ be an order bounded disjointness preserving $KB$/$\sigma$-$KB$ operator between locally solid lattices $(X,\varsigma)$ and $(Y,\tau)$.
If $|S|\le|T|$ then $S$ is $KB$/$\sigma$-$KB$.
\end{theorem}

\begin{proof}
Take a $\varsigma$-bounded increasing net/sequence $x_\alpha$ in $X_+$. Then $T(x_\alpha-x)\convtau 0$ for some $x\in X$. It follows from
$$
   |S(x_\alpha-x)|\le|S||x_\alpha-x|\le|T||x_\alpha-x|=|Tx_\alpha-Tx|\convtau 0
$$ 
that $Sx_\alpha\convtau Sx$, as desired.
\end{proof}

Since $0\le S\le T$ with a lattice homomorphism $T$ implies that $S$ is also a lattice homomorphism, 
the next result is a direct consequence of  Theorem \ref{KB disj pres dominated property}.

\begin{cor}\label{KB lattice homomorphism dominated property}
Let $T$ be a $KB$/$\sigma$-$KB$ lattice homomorphism between locally solid lattices $(X,\varsigma)$ and $(Y,\tau)$.
Then each $S$ satisfying  $0\le S\le T$ is also a $KB$/$\sigma$-$KB$ lattice homomorphism.
\end{cor}

\begin{theorem}\label{quasi $KB$ dominated property}
Let $T$ be a positive quasi $KB$ operator  between locally solid lattices $(X,\varsigma)$ and $(Y,\tau)$.
Then each operator $S:X\to Y$ with $0\le S\le T$ is also quasi $KB$.
\end{theorem}

\begin{proof}
Let $0\le S\le T$ and $x_\alpha$ be an increasing $\varsigma$-bounded net in $X_+$. Since $T\ge 0$ then $Tx_\alpha\uparrow$.
Since $T$ is quasi $KB$ then $Tx_\alpha$ is $\tau$-Cauchy. Take a $U\in\tau(0)$ and a solid neighborhood $V\in\tau(0)$
with $V-V\subseteq U$. There exists $\alpha_0$ satisfying $T(x_\alpha-x_\beta)\in V$ for all $\alpha,\beta\ge\alpha_0$. 
In particular, $T(x_\alpha-x_{\alpha_0})\in V$ for all $\alpha\ge\alpha_0$. Since $0\le S\le T$ and $V$ is solid,
$S(x_\alpha-x_{\alpha_0})\in V$ for all $\alpha\ge\alpha_0$. Thus, we obtain
$$
  Sx_\alpha-Sx_\beta=S(x_\alpha-x_{\alpha_0})-S(x_\beta-x_{\alpha_0})\in V-V\subseteq U
  \eqno(5)
$$ 
for all $\alpha,\beta\ge\alpha_0$. Since $U\in\tau(0)$ was taken arbitrary, $(5)$ implies that $Sx_\alpha$ is also $\tau$-Cauchy.
\end{proof}

\begin{theorem}\label{quasi Levi dominated property}
Let $T$ be a positive quasi Levi operator from a locally solid lattice $(X,\varsigma)$ to a vector lattice $Y$.
Then each operator $S:X\to Y$ with $0\le S\le T$ is also quasi  Levi.
\end{theorem}

\begin{proof}
Let $T:X\to Y$ be quasi Levi and $x_\alpha$ be a positive $\varsigma$-bounded increasing net in $(X, \varsigma)$.
Then $Tx_\alpha$ is $o$-Cauchy in $Y$. So, there exists a net $z_\beta\downarrow$ in $Y$ such that, for each $\beta$,
there exists $\alpha_\beta$ such that $|Tx_{\alpha_1}-Tx_{\alpha_2}|\le z_\beta$ for all $\alpha_1, \alpha_2\ge \alpha_\beta$.
Choosing $\alpha_1, \alpha_2\ge \alpha_\beta$ for a fixed $\alpha_\beta$ we have
$$
   Sx_{\alpha_1}-Sx_{\alpha_2}\le S(x_{\alpha_1}-x_{\alpha_\beta})\le 
   T(x_{\alpha_1}-x_{\alpha_\beta}) \le z_\beta,
$$
and similarly
$$
   Sx_{\alpha_2}-Sx_{\alpha_1}\le S(x_{\alpha_2}-x_{\alpha_\beta})\le 
   T(x_{\alpha_2}-x_{\alpha_\beta}) \le z_\beta .
$$
Thus $|Sx_{\alpha_1}-Sx_{\alpha_2}|\le z_\beta$ for all $\alpha_1, \alpha_2\ge \alpha_\beta$ showing that the operator $S$ is also quasi Levi.
\end{proof}

Any order dense sublattice is regular by \cite[Thm.1.27]{AB1}.
Also a locally solid lattice $(X,\varsigma)$ is 
a regular sublattice of its $\varsigma$-completion $\hat{X}$ iff every $\varsigma$-Cauchy net $x_\alpha$ in $X_+$ such that $x_\alpha\convo 0$ in $X$ 
satisfies $x_\alpha\convvars 0$ \cite[Thm.2.41]{AB1}.

\begin{lem}\label{sveta1}
Let $Z$ be a regular sublattice of a vector lattice $X$ and $T$ be an $o\tau$-continuous operator  from $X$ to a topological vector space $(Y,\tau)$.
Then  the restriction $T|_Z: Z\to Y$ is also $o\tau$-continuous.
\end{lem}

\begin{proof}
Under the assumptions of the lemma, for each net $z_\alpha$ in $Z$, $z_\alpha\convo 0$ in $X$ if $z_\alpha\convo 0$ in $Z$, by \cite[Lem.2.5]{GTX}.
Therefore the result follows directly from the definition of an $o\tau$-con\-ti\-nuous operator.
\end{proof}

\noindent
Similarly for $uo$-null nets. Let $Z$ be a regular sublattice of $X$ then $x_\alpha\convuo 0$ in $Z$ iff $x_\alpha\convuo 0$ in $X$
by \cite[Thm.3.2]{GTX}, and hence, for a $uo\tau$-continuous operator $T$ from a vector lattice $X$ to a topological vector space $(Y,\tau)$, 
the restriction $T|_Z: Z\to Y$ is also $uo\tau$-continuous.

\begin{lem}\label{sveta2}
Let $T$ be a positive Lebesgue operator  from a vector lattice $X$ to a locally solid lattice $(Y,\tau)$.
If $Z$ is a regular sublattice of $X$, then the restriction $T|_Z: Z\to Y$ is $o\tau$-continu\-ous.
\end{lem}

\begin{proof}
$T$ is $o\tau$-continuous by Lemma \ref{PC1}. It follows from Lemma \ref{sveta1} that  $T|_Z: Z\to Y$ is $o\tau$-continu\-ous.
\end{proof}

\noindent
Since a vector lattice $X$ is order dense in $X^\delta$, and hence is regular in $X^\delta$, the next result follows immediately from Lemma \ref{sveta2}.

\begin{prop}\label{sveta3}
If an operator $T$ from the Dedekind completion $X^\delta$ of a vector lattice $X$ to  a locally solid lattice $(Y,\tau)$ is Lebesgue then  the operator $T|_X$ is $o\tau$-continu\-ous.
\end{prop}

Each order continuous lattice homomorphism $T$ from a Fatou lattice $(X,\varsigma)$ to a Lebesgue lattice $(Y,\tau)$ is $\varsigma\tau$-continuous on order intervals of $X$ by \cite[Thm.4.23]{AB1}. 
Since $T$ is also Lebesgue under the given conditions, the following can be considered as a generalization of \cite[Thm.4.23]{AB1}.

\begin{theorem}\label{restricted to order bounded subsets}
Let $(X,\varsigma)$ be a Fatou lattice, $(Y,\tau)$ a Lebes\-gue lattice, and $T:X\to Y$ a Lebesgue lattice homomorphism.  
Then $T$ is $\varsigma\tau$-continuous when restricted to order bounded subsets in $X$.
\end{theorem}

\begin{proof}
Let ${\cal N}$ be a base of solid neighborhoods for $\tau(0)$. Since $T$ is a Lebes\-gue lattice homomorphism, the family 
$\{T^{-1}(V)\}_{V\in{\cal N}}$ is a base of solid neighborhoods of zero for a Lebesgue (not necessary Hausdorff) topology $\zeta$ on $X$. 
Since $\varsigma$ is Fatou and $\zeta$ is Lebesgue, then on any order bounded subset 
of $X$ the topology induced by $\varsigma$ is finer than the topology induced by $\zeta$ due to Theorem 4.21 of \cite{AB1}.
Clearly, $T:X\to Y$ is $\zeta\tau$-continuous. Thus, $T$ is also $\varsigma\tau$-continuous on order bounded subsets of $X$.
\end{proof}
 
Recall that, for a locally solid lattice $(X,\varsigma)$, a vector $\hat{v}\in \hat{X}_+$ is called an {\em upper element} of $X$ whenever there exists an increasing net $u_\alpha$ in $X_+$ 
such that $u_\alpha\stackrel{\hat{\varsigma}}{\to} \hat{v}$ (cf. \cite[Def.5.1]{AB1}). In a metrizable locally solid lattice $(X,\varsigma)$ upper elements can be described by sequences: 
a vector $\hat{v}\in \hat{X}_+$ is an upper element of $X$ iff there exists an increasing sequence $u_n$ in $X_+$ such that $u_n\stackrel{\hat{\varsigma}}{\to} \hat{v}$ (cf. \cite[Thm.5.2]{AB1}). 
For a metrizable $\sigma$-Lebesgue lattice $(X,\varsigma)$ its topological completion $(\hat{X}, \hat{\varsigma})$ is also $\sigma$-Lebesgue (cf. \cite[Thm.5.36]{AB1}). 

\begin{theorem}\label{sveta4}
Let $(X,\varsigma)$ be a metrizable locally solid lattice, $(Y,\tau)$ a $\sigma$-Lebesgue lattice, and $T:(X,\varsigma)\to(Y,\tau)$ a positive $\varsigma\tau$-continuous $\sigma$-Lebesgue operator. 
Then the unique extension $\hat{T}:(\hat{X}, \hat{\varsigma}) \to (\hat{Y}, \hat{\tau})$ is positive, $\hat{\varsigma}\hat{\tau}$-con\-ti\-nu\-ous, and $\sigma$-Lebesgue. 
\end{theorem}

\begin{proof}
Clearly, $\hat{T}$ is positive and $\hat{\varsigma}\hat{\tau}$-continuous. To show that $\hat{T}$ is $\sigma$-Lebesgue, assume $\hat{x}_n\downarrow 0$ in $(\hat{X},\hat{\varsigma})$. 
Let $W\in\tau(0)$. We have to show $\hat{T}\hat{x}_n\in\overline{W}$ for large enough $n$. Take a solid neighborhood $W_1\in\tau(0)$ with $W_1 + W_1 \subseteq W$ and choose 
a neighborhood $V\in\varsigma(0)$ with $T(V) \subseteq W_1$. Let $\{V_n\}_{n=1}^{\infty}$ be a base for $\varsigma(0)$ satisfying $V_{n+1}+ V_{n+1} \subseteq V_n$ 
for all natural $n$ and also $V_1+V_1\subseteq V$. At this point we borrow the second paragraph of the proof of \cite[Thm.5.36]{AB1}.
 
For each $n$, let $\hat{v}_n\in\hat{X}$ be an upper element of $X$ such that $\hat{x}_n\le\hat{v}_n$ and $\hat{v}_n-\hat{x}_n\in\overline{V}_n$; see \cite[Thm.5.4]{AB1}.
Put $\hat{w}_n:=\wedge^n_{i=1}\hat{v}_n$ for all $n$ and note that each $\hat{w}_n$ is an upper element of $X$, $\hat{w}_n\downarrow$ and 
$\hat{w}_n-\hat{x}_n\in\overline{V}_n$. Since $\hat{x}_n\downarrow 0$, it follows from \cite[Thm.2.21(e)]{AB1} that $\hat{w}_n\downarrow 0$ also holds in $\hat{X}$. 
Thus, we can assume without loss in generality that $\hat{x}_n$ is a sequence of upper elements of $X$. 

Take $u_n\in X_+$ with $u_n \le \hat{x}_n$ and $\hat{x}_n - u_n \in \overline{V}_n$.  Application of \cite[Thm.2.21(e)]{AB1} to the nets 
$\hat{x}_n\downarrow 0$ and $z_n:=\wedge^n_{i=1}u_n\downarrow $ with $\hat{x}_n-z_n\convvarsh 0$  gives $z_n\downarrow 0$ in $X$.
As $T$ is $\sigma$-Lebesgue, $Tz_n\convtau 0$. So, there exists some $n_1$ such that
$$
   \hat{T}z_n=Tz_n\in W_1\subseteq\overline{W}_1\ \ \ \ \ (\forall n\ge n_1).
   \eqno(6)
$$
Observe that 
$$
   0\le  \hat{x}_n - z_n = \hat{x}_n - \wedge^n_{i=1} u_i = \vee^n_{i=1}(\hat{x}_n - u_i) \le
   \vee^n_{i=1}(\hat{x}_i - u_i)\le
$$
$$
   \sum^n_{i=1}(\hat{x}_i - u_i) \in\overline{V}_1 + \overline{V}_2 + \dotsc + \overline{V}_n \subseteq \overline{V} \ \ \ (\forall n\in\mathbb{N}).
   \eqno(7)
$$
It follows from $(6)$, $(7)$, and $\hat{T}(\overline{V}) \subseteq \overline{W}_1$ that
$$
    \hat{T}(\hat{x}_n) \le \hat{T}(z_n)+\hat{T}(\hat{x}_n - z_n ) \in \overline{W}_1 + \overline{W}_1 \subseteq \overline{W} \ \ \ (\forall \ n\ge n_1).
$$
Since $W\in\tau(0)$ was arbitrary, $\hat{T}$ is $\sigma$-Lebesgue.
\end{proof}

\noindent
Since Theorem 5.36 of \cite{AB1} follows from Theorem \ref{sveta4} when $(X,\varsigma)$ coincides with $(Y,\tau)$ and $T$ is the identity operator on $X$,
Theorem \ref{sveta4} can be also considered as a generalization of \cite[Thm.5.36]{AB1}.

One more question deserves being mentioned.  Let $T$ be Lebes\-gue, $o\tau$-continuous,  $o\tau$-bounded, $o\tau$-compact, $KB$, or Levi. Does $|T|$ satisfy the same property?
The answer is positive, when $T$ maps a vector lattice $X$ to a  Dedekind complete locally solid lattice $(Y,\tau)$, in the following cases:
\begin{enumerate}
\item[$(*)$] \  
for a  Lebes\-gue/$o\tau$-continuous order continuous operator $T$, when $(Y,\tau)$ is Lebesgue, since $|T|$ is order continuous 
by \cite[Thm.1.56]{AB2} and hence $x_\alpha\downarrow 0$/$x_\alpha\convo 0$ in $X$ imply $|T|x_\alpha\convo 0$ and (since 
$(Y,\tau)$ is Lebesgue)  $|T|x_\alpha\convtau 0$;
\item[$(**)$] \  
for an $o\tau$-bounded operator $T$, by Proposition \ref{regular are otau-bounded}.
\end{enumerate} 

\begin{rem}
{\em
\begin{enumerate}
Let $(X,\varsigma)$ and $(X,\tau)$ be locally solid lattices.
\item[$a)$] \
Recall that $(Y,\tau)$ is said to be {\em boundedly order-boun\-ded} (BOB) whenever 
increasing $\tau$-bounded nets in $Y_+$ are order boun\-ded in $Y$ (cf. \cite[Def.3.19]{Tay1}).

Assume $(Y,\tau)$ to be Dedekind complete and BOB, Then each positive quasi $KB$ operator from $(X,\varsigma)$ to $(Y,\tau)$ is quasi Levi. Let $x_\alpha$ be an increasing $\varsigma$-bounded net in $X_+$. 
Then $Tx_\alpha$ is a $\tau$-Cauchy increasing net in $Y_+$. Then $Tx_\alpha$ is $\tau$-bounded and hence $o$-bounded since $(Y,\tau)$ is BOB. 
The Dedekind completeness of $Y$ implies that $Tx_\alpha\uparrow y$ for some $y\in Y$. Therefore $Tx_\alpha$ is $o$-Cauchy as desired.
\item[$b)$] \
Recall that $(X, \varsigma)$ satisfies the {\em $B$-property} (cf. \cite[Def.3.1]{Tay1}) whenever the identity operator $I$ on $X$ is quasi $\sigma$-$KB$.
Let $T:(X,\varsigma)\to(Y,\tau)$ be a continuous operator and assume the topology $\varsigma$ in $(X,\varsigma)$ to be minimal. Then $T$ is Lebesgue and quasi $KB$.
The first part follows as $(X, \varsigma)$ is Lebesgue by \cite[Thm.7.67]{AB1} and hence $T: (X, \varsigma) \to (Y, \tau )$ is  Lebesgue.
The second part follows as $(X, \varsigma)$ satisfies the $B$-property \cite[Prop.3.2]{Tay1} and hence $T$ is $\sigma$-$KB$, and thus $T$ is quasi $KB$
by Proposition \ref{quasi-KB-vs-sigma-quasi-KB}. 
\end{enumerate}
\noindent
Similarly, each regular operator from a vector lattice $X$ to a locally solid lattice $(Y, \tau)$ with minimal topology $\tau$ is quasi Lebesgue by \cite[Cor.3.3]{Tay1}.
}
\end{rem}

\section{Operators between Banach lattices}

In this section we investigate operators whose domains and ranges are Banach lattices. 
It is clear that every continuous operator from an order continuous Banach lattice to a topological vector space is Lebesgue. 
The following proposition generalizes \cite[Prop.2]{AEEM1}.

\begin{prop}\label{Prop.2 from 1AEEM}
An operator $T$ from a $\sigma$-order continuous Banach lattice $X$ to a metrizable topological vector space $(Y,\tau)$ is continuous iff $T$ is $\sigma{}o\tau$-continuous.
\end{prop}

\begin{proof}
The necessity is trivial.

For the sufficiency: let $T$ be $\sigma{}o\tau$-continuous, $x_\alpha\convnX x$, and $x_{\alpha_\beta}$ be a subnet of the net $x_\alpha$.
By $a)$ of Remark \ref{sequential}, there exists a sequence $\beta_n$ of indices with $\|x_{\alpha_{\beta_n}}-x\|_X\to 0$.  Since $X$ is a Banach lattice, 
there is a subsequence $x_{\alpha_{\beta_{n_k}}}$ of $x_{\alpha_{\beta_n}}$ satisfying $x_{\alpha_{\beta_{n_k}}}\convo x$ in $X$ (see \cite[Thm.VII.2.1]{V}). 
Since $T$ is $\sigma{}o\tau$-continuous then $Tx_{\alpha_{\beta_{n_k}}}\convtau Tx$. Remark \ref{sequential} implies $Tx_\alpha\convtau Tx$, as required.
\end{proof}

In the next proposition we transfer the $\sigma$-condition from the domain $X$ of $T$ to the operator $T$. 

\begin{prop}\label{sigma-otau-cont to cont}
Each $\sigma{}o\tau$-continuous operator $T$ from a Banach lattice $X$ to a Banach lattice $Y$ is continuous. 
\end{prop}

\begin{proof}
It is sufficient to show that $\|Tx_n\|\to 0$ for every norm-null sequence $x_n$ in $X$.
Let $\|x_n\|\to 0$ in $X$. For proving $\|Tx_n\|\to 0$, in view of $b)$ of Remark \ref{sequential},
it is enough to show that, for each subsequence $Tx_{n_k}$ of $Tx_n$, there exists a sequence $n_{k_i}$ of indices with 
$\|Tx_{n_{k_i}}\|\to 0$. Let $Tx_{n_k}$ be a subsequence of $Tx_n$. By \cite[Thm.VII.2.1]{V}, $x_{n_k}$ has 
a further subsequence $x_{n_{k_i}}$ with $x_{n_{k_i}}\convo 0$, and since $T$ is $\sigma{}o\tau$-continuous, 
$\|Tx_{n_{k_i}}\|\to 0$ as desired.
\end{proof}

\begin{theorem}\label{see Thm. 5.28 from AB2}
Let $T:X\to Y$ be an order bounded operator between Banach lattices $X$ and $Y$, the norms in $X'$ and $Y$ be order continuous, and $Y$ be weakly sequentially complete. Then $T$ is a quasi $KB$-operator.
\end{theorem}

\begin{proof}
Since $Y$ has order continuous norm, $Y$ is Dedekind complete, and hence the operator $T$ is regular. Therefore, we may assume $T\ge 0$. 
By Proposition \ref{quasi-KB-vs-sigma-quasi-KB}, it suffices to prove that $T$ is quasi $\sigma$-$KB$.  Let $x_n$ be a $\|\cdot\|_X$-bounded increasing 
sequence in $X_+$.  Then $Tx_n$ is $\|\cdot\|_Y$-bounded and $0\le Tx_n\uparrow$. It follows from \cite[Thm.5.28]{AB2} that $Tx_n$ has a weak 
Cauchy subsequence $Tx_{n_k}$. Then $Tx_n$ is weak Cauchy and hence $Tx_n\convw y$ for some $y\in Y$. 
Since $Tx_n\uparrow$, it follows from \cite[Thm.3.52]{AB2} that $Tx_n\convnY y$, and hence $Tx_n$ is $\|\cdot\|_Y$-Cauchy as required.
\end{proof}

As $c_0$ is not a $KB$-space, the identity operator $I:c_0\to c_0$ is not quasi $KB$
by Remark \label{c_00} $a)$; $d)$. Thus, the weak sequential completeness
of $E$ is essential in Theorem \ref{see Thm. 5.28 from AB2}.

\begin{exam}\label{c_w(R)} 
{\em 
Let $X=(c_\omega(\mathbb{R}),\|\cdot\|_\infty)$ be the Banach lattice of all bounded $\mathbb{R}$-valued functions on $\mathbb{R}$ such that each $f\in X$ differs 
from a constant say $a_f$ on a countable subset of $\mathbb{R}$. It is easy to see that $X$ is Dedekind $\sigma$-complete yet not Dedekind complete.

We define a positive operator $T:X\to X$ as follows.  
$Tf$ is the constant function $a_f\cdot\mathbb{I}_{\mathbb{R}}\in X$, for which the set $\{d\in\mathbb{R}: f(d)\ne a_f\}$ is countable. 
Clearly, $T$ is quasi $KB$ and hence quasi Lebesgue.
\begin{enumerate}
\item[$(i)$] \ 
For the sequence $f_n:=\mathbb{I}_{\{1,2,\ldots n\}}\in X$ we have that $f_n\convo\mathbb{I}_{\mathbb{N}}\in X$, yet $\|f_n-\mathbb{I}_{\mathbb{N}}\|_\infty=1$ for all $n$.
Thus, the norm in $X$ is not $\sigma$-order continuous. In particular, Proposition \ref{Prop.2 from 1AEEM} is not applicable to $T$.
\item[$(ii)$] \ 
$T$ is rank-one continuous $($therefore quasi $KB$, compact, and $o\tau$-compact$)$ operator in $X$. 
Let $f_n\convo 0$. Since for every $\varepsilon>0$ there exists $n_\varepsilon$ such that 
the set $\cup_{n\ge n_\varepsilon}\{d\in\mathbb{R}: |f_n(d)|\ge\varepsilon\}$ is countable, $\|Tf_n\|_\infty<\varepsilon$ for all $n\ge n_\varepsilon$.
Thus, $T$ is $\sigma{}o\tau$-continuous and hence $\sigma$-Lebesgue with respect to the norm topology  on $X$.
\item[$(iii)$] \ 
Operator $T$ is not Lebesgue. Indeed, for the net $f_\alpha:=\mathbb{I}_{\mathbb{R}\setminus\alpha}\in X$ indexed by the family $\Delta$ of all finite subsets 
of $\mathbb{R}$ ordered by inclusion, we have $f_\alpha\downarrow 0$ yet $\|Tf_\alpha\|_\infty=\|\mathbb{I}_{\mathbb{R}}\|_\infty=1$ for all $\alpha\in\Delta$.
The same argument shows that $T$ is not weakly Lebesgue. Since $Tf_k\ne 0$ for at most one $k$ for each disjoint sequence $f_n$ in $X$, operator
$T$ is $M$-weakly compact. It should be clear that $T$ is not $L$-weakly compact.
\item[$(iv)$] \ 
$T$ is $KB$. Indeed, let $X\ni f_\alpha\uparrow$ and $\|f_\alpha\|\le M\in\mathbb{R}$. Then $a_{f_\alpha}\uparrow\le M$.
Take $f:=a\cdot\mathbb{I}_{\mathbb{R}}\in X$ for $a=\sup_\alpha a_{f_\alpha}\in\mathbb{R}$. Clearly, $\|T(f-f_\alpha)\|\downarrow 0$, as desired. 
\item[$(v)$] \ 
Clearly, $T$ is a Dunford-Pettis lattice homomorphism. By Corollary \ref{KB lattice homomorphism dominated property}, the span of $[0,T]$ is a
vector sublattice of $KB$-operators in the ordered space $L_r(X)$ of all regular operators in $X$.
\item[$(vi)$] \ 
The vector lattice $X$ is $\tau$-laterally $\sigma$-complete yet not $\tau$-laterally complete with respect to the norm topology on $X$. 
\end{enumerate}
}
\end{exam}

Observe that a $KB$-operator $T:\ell_1\to \ell_\infty$ defined by $[T(a)]_k:=\sum_{n=1}^{\infty}a_n$ for all $k\in\mathbb{N}$ is neither $L$- nor $M$-weakly compact (cf., e.g., \cite[p.322]{AB2}).
In the present paper we do not investigate conditions under that $KB$-operators are (weakly) Lebesgue. The following lemma generalizes the observation that the identity operator 
in a Banach lattice $X$ is quasi $KB$ iff $X$ is a $KB$-space.

\begin{lem}\label{Ghoussoub-Johnson 1}
If $c_0$ does not embed in a Banach space $Y$, then every continuous operator $T$ from any Banach lattice $X$ to $Y$ is a quasi $KB$ operator.
\end{lem}

\begin{proof}
It follows from  \cite[Thm.4.63]{AB2} that $T=S\circ Q$, where $Q$ is a lattice homomorphism from $X$ to a $KB$-space $Z$ 
(and hence $Q$ is a quasi $KB$-operator by $d)$ of Remark \ref{c_00}) and $S$ is a continuous operator from $Z$ to $Y$. 
Now apply the fact that the class of quasi $KB$-operators is closed under the left composition with continuous operators. 
\end{proof}

\begin{theorem}\label{Ghoussoub-Johnson 2}
Let $X$ be a Banach lattice and $Y$ a Banach space. TFAE.
\begin{enumerate}
\item[$(i)$] \
Every continuous operator $T:X\to Y$ is quasi $KB$.
\item[$(ii)$] \
$c_0$ does not embed in $Y$.
\end{enumerate}
\end{theorem}

\begin{proof}
$(ii)\Rightarrow(i)$ is Lemma \ref{Ghoussoub-Johnson 1}.\

$(i)\Rightarrow(ii)$ It suffices to prove that any embedding  $J:c_0\to Y$ is not a quasi $KB$-operator.
Assume by way of contradiction that such a $J$ is quasi $KB$ and let $x_n=\sum_{k=1}^n e_k\in c_0$. 
The increasing sequence $x_n$ is norm bounded, yet the sequence $Jx_n$ is not $\|\cdot\|_Y$-Cauchy because  
$$
   \|Jx_n-Jx_m\|_Y=\left\|J\sum_{k=m}^n e_k\right\|\ge \frac{1}{\|J^{-1}|_{J(c_0)}\|}>0 \ \ \ \ (n\ne m).
$$ 
The obtained contradiction completes the proof. 
\end{proof}

Among other things, Theorem \ref{Ghoussoub-Johnson 2} asserts that, in the Banach space setting, 
the concept of continuous quasi $KB$-opera\-tor does not depend on the domain $X$
and is completely determined by the property that the range space $Y$ not containing an isomorphic copy of $c_0$.
In light of this observation, the last sentence in item $d)$ of Remark \ref{c_00} is reduced to the condition $1\le p<\infty$ 
under which $\overline{c_{00}}^{\|\cdot\|_p}=\ell_p$ is a $KB$-space and to the condition $1\le l\le p$ under which 
the identity operator $I$ from $(c_{00},\|\cdot\|_l)$ to $(c_{00},\|\cdot\|_p)$ is continuous.

Now we pass to a discussion of $KB$ operators.

\begin{theorem}\label{when KB}
Let  $T:X\to Y$ be regular operator between Banach lattices. If  $X'$ has order continuous norm and $c_0$ does not embed in $Y$ then $T$ is KB. 
\end{theorem}

\begin{proof}
Without lost of generality, we suppose $T\ge 0$. By \cite[Thm.4.63]{AB2}, $T$ admits a factorization $T = S\circ Q$, 
where $Z$ is a $KB$-space and $Q:X\to Z$ is a lattice homomorphism.  Let $x_\alpha$ be an increasing norm bounded net in $X_+$. 
Since $Q\ge 0$ then $Qx_\alpha$ is an increasing norm bounded net in $Z_+$, and since $Z$ is a $KB$-space, there exist $z\in Z$ 
with $\|Qx_\alpha-z\|\to 0$ and hence $\|Tx_\alpha-Sz\|\to 0$ as desired.
\end{proof}

The following theorem gives conditions under which each positive weakly compact operator from a Banach lattice $X$ to an arbitrary $KB$-space $Y$ is $\sigma{}o\tau$-continuous.

\begin{theorem}\label{AB2 Thm 5.42}
For a Banach lattice $X$ TFAE. 
\begin{enumerate}
\item[$(i)$] \  
The image $i(X)$ under the natural embedding $i:X\to X''$ is a regular sublattice of $X''$.
\item[$(ii)$] \  
Each positive weakly compact operator $T$ from $X$ to a $KB$-space $Y$ is $\sigma{}o\tau$-continuous.
\item[$(iii)$] \  
Each positive compact operator $T$ from $X$ to a $KB$-space $Y$ is $\sigma{}o\tau$-continuous.
\item[$(iv)$] \  
$T^2$ is $\sigma{}o\tau$-continuous for every positive weakly compact operator $T$ in $X$.
\item[$(v)$] \  
$T^2$ is $\sigma{}o\tau$-continuous for every positive compact operator $T$ in $X$.
\item[$(vi)$] \  
$T^2$ is Lebesgue for every positive compact operator $T$ in $X$.
\end{enumerate}
\end{theorem}

\begin{proof}
$(i)\Rightarrow(ii)$ \
Let $T$ be a positive weakly compact operator from $X$ to a $KB$-space $Y$.
By \cite[Thm.5.42]{AB2}, $T=R\circ S$ with both $S:X\to Z$ and $R:Z\to Y$ are positive and $Z$ is a reflexive Banach lattice.
Let $x_\alpha\convo 0$ in $X$. Since $i(X)$ is regular in $X''$ then $i(x_\alpha)\convo 0$ in $X''$ (e.g., by \cite[Thm.1.20]{AB1}).
Since $S'':X''\to Z''$ is order continuous (e.g., by \cite[Thm.1.73]{AB2}), $S''(i(x_\alpha))\convo 0$ in $Z''$, and hence $Sx_\alpha=S''(i(x_\alpha))\convo 0$ in $Z$. 
The order continuity of norm in $Z$ implies $\|Sx_\alpha\|_Z\to 0$ and hence $\|Tx_\alpha\|_Y=\|R\circ S(x_\alpha)\|_Y\to 0$,
showing that $T$ is $\sigma{}o\tau$-continuous.

$(ii)\Rightarrow(iii)$ and $(iv)\Rightarrow(v)\Rightarrow(vi)$ \
are trivial.

$(iii)\Rightarrow(i)$ \
If $i(X)$ is not regular in $X''$, then (e.g., by \cite[Thm.3.12]{AB1}) there exists a positive functional $f\in X'$ which is not order continuous.
Thus $f:X\to\mathbb{R}$ is positive compact yet not $\sigma{}o\tau$-continuous as $f(x_\alpha)$ does not converge to $0$ in the $KB$-space $Y=\mathbb{R}$ for some net 
with $x_\alpha\convo 0$ in $X$. The obtained contradiction proves that $i(X)$ must be regular in $X''$.

$(i)\Rightarrow(iv)$ \
The proof is similar to the proof of $(i)\Rightarrow(ii)$ above with the only difference that the factorization of $T^2$ by \cite[Cor.5.46]{AB2} must be used 
instead of the factorization of $T$ by \cite[Thm.5.42]{AB2}.

$(vi)\Rightarrow(i)$ \
The proof is again similar to the proof of $(iii)\Rightarrow(i)$ above. If $i(X)$ is not regular in $X''$, then there exists a positive functional $f\in X'$ which is not order continuous. 
Take any $z\in X_+$ with $f(z)=1$ and define a positive compact operator $T$ in $X$ by $Tx:=f\otimes z$. As $f$ is not order continuous, there is a net $x_\alpha$ with 
$x_\alpha\downarrow 0$ and $f(x_\alpha)\ge 1$ for all $\alpha$. Therefore 
$$
   T^2(x_\alpha)=T(f(x_\alpha)z)=f(x_\alpha)Tz=f(x_\alpha)z\ge z>0 \ \ \ (\forall \alpha)
$$ 
violating that $T^2$ is Lebesgue. The obtained contradiction completes the proof.
\end{proof}

\begin{prop}\label{when otau bounded is KB}
Let  $T:X\to Y$ be a $o\tau $-bounded operator from a Banach lattice $X$ to a Banach space $Y$. 
If $c_0$ does not embed in $Y$, then $T$ is $\sigma$-KB. 
\end{prop}

\begin{proof}
By \cite[Thm.3.4.6]{MN}, $T$ factors over a $KB$-space $Z$ as $T=S\circ Q$, where $Q$ is a lattice homomorphism $Q: X \to Z$, and $S:Z\to Y$ is continuous.
Let $x_n$ be a norm bounded increasing sequence in $X_+$. Since $Q$ is positive, $Qx_n$ is positive norm bounded increasing sequence in the $KB$-space $Z$. 
Thus, $Qx_n$ is norm convergent in $Z$. As $S$ is continuous, $Tx_n=SQx_n$ is norm convergent and hence $T$ is $\sigma$-KB. 
\end{proof}

\begin{prop}
Let  $T:X\to Y$ be a continuous operator from a Banach lattice $X$ to a Banach space $Y$. 
If either $T''$ is order-weakly compact or the norm in $X$ is order continuous and $T$ 
does not preserve a sublattice isomorphic to $c_0$, then $T$ is $\sigma$-KB. 
\end{prop}

\begin{proof}
The operator $T:X\to Y$ factors over a $KB$-space $Z$ as $T = S\circ Q$, where $Q: X\to Z$ is a lattice homomorphism
and $S : Z\to Y$ continuous by \cite[Thm.3.5.8]{MN}. Let $x_n$ be a norm bounded increasing sequence in $X_+$, then 
$Qx_n$ is also norm bounded increasing in $Z_+$. As $Z$ is a $KB$-space, $Qx_n$ is norm convergent in $Z$. 
Hence $(S\circ Q)x_n$ is convergent in $Y$ and $T$ is a $\sigma$-KB operator. 
\end{proof}

Remark that if $T:X\to Y$ is a continuous operator from a Banach lattice $X$ to a Banach space $Y$ and $T$ preserves no subspace (or no sublattice) 
isomorphic to $c_0$, then $Tx_n$ is norm convergent for each increasing norm bounded sequence in $X_+$ by \cite[Thm.3.4.11]{MN} and therefore  $T$ is $\sigma$-KB. 

It follows from \cite[Thm.1]{Wick} that every $uo\tau$-continu\-ous operator $T$ from an atomic Banach lattice $X$ to a Banach space $Y$ 
is a Dunford-Pettis operator. By Theorem 5.57 of \cite{AB2}, a continuous operator $T:X\to Y$ from a Banach lattice $X$ to a Banach space $Y$ is 
$o$-weakly compact iff, for each order bounded disjoint sequence $x_n$ in $X$, we have $\|Tx_n\|\to 0$.

\begin{prop}\label{pos w-comp is levi}
Each positive weakly compact operator $T$ between Banach lattices $X$ and $Y$ is Levi. 
\end{prop}

\begin{proof}
Let $x_\alpha$ be a norm bounded increasing net in $X_+$. Weak compactness of $T$ ensures that 
$Tx_{\alpha_\beta}\convw y$ for some $y\in Y$ and some subnet $x_{\alpha_\beta}$ of $x_\alpha$.
Since $Tx_{\alpha_\beta}\uparrow$ then $Tx_{\alpha_\beta}\convw y$ implies  $Tx_{\alpha_\beta}\uparrow y$
and hence $Tx_\alpha\uparrow y$ which means that $T$ is a Levi operator.
\end{proof}

\begin{rem} {\em
\begin{enumerate}
\item[$a)$] \  
It follows directly from \cite[Cor.3.4.12]{MN} that every continuous operator from a Banach lattice $X$ to a Banach space $Y$ that does not contain $c_0$ is quasi $KB$.
\item[$b)$] \  
Let $X$ be a Banach lattice and $Y$ be a Banach space. Then every continuous operator $T:X\to Y$ that does not preserve $c_0$
is quasi-KB by \cite[Thm.3.4.11]{MN}. Under the same settings, it was observed in \cite[3.4.E4, p. 203]{MN} that
if $T''({\cal B}_X) \subseteq Y$, where ${\cal B}_X$ is the band generated by $X$ in $X''$, then $T$ does not preserve a subspace 
isomorphic to  $c_0$. Note that $b$-weakly compact operators satisfy $T''({\cal B}_X) \subseteq Y$, whenever $X$ has order continuous norm. 
It follows that under these conditions, each $b$-weakly compact operator is quasi KB (see \cite[Prop 2.11]{AAT1}).
\end{enumerate}
}
\end{rem}

{
}
\end{document}